\documentclass[12pt, twoside, reqno] {amsart}

\usepackage{amsfonts}
\usepackage{amssymb}
\usepackage{latexsym}
\usepackage{epsf,graphicx}
\usepackage{epsf,graphicx,latexsym,%
}

\usepackage{xcolor}

\newcommand{\be} {\begin{eqnarray}}
\newcommand{\ee} {\end{eqnarray}}
\newcommand{\bep} {\begin{eqnarray*}}
\newcommand{\eep} {\end{eqnarray*}}

\newcommand {\Hol}{\mathop{\rm Hol}\nolimits}

\newcommand {\Id}{\mathop{\rm Id}\nolimits}

\renewcommand {\Re}{\mathop{\rm Re}\nolimits}

\newcommand {\JJ}{\mathcal{J}}

\newcommand {\A}{\mathcal{A}}
\newcommand {\BB}{\mathcal{B}}

\newcommand{\R}{{\mathbb R}}

\newcommand{\C}{{\mathbb C}}

\newcommand {\D}{\mathbb{D}}

\newtheorem{remar}{Remark}[section]
\newtheorem{examp}{Example}[section]
\newtheorem{defin}{Definition}[section]
\newtheorem{corol}{Corollary}[section]
\newtheorem{propo}{Proposition}[section]
\newtheorem{theorem}{Theorem}[section]
\newtheorem{lemma}{Lemma}[section]

\newtheorem{conj}{Conjecture}
\newtheorem{quest}{Question}

\newcommand{\rema}{\begin{remar}\rm}
\newcommand{\erema}{$\blacktriangleright$\end{remar}}

\newcommand{\exa}{\begin{examp}\rm}
\newcommand{\eexa}{$\blacktriangleright$\end{examp}}

\def\lwvec(#1 #2){\linewd 0.1
           \lvec(#1 #2)
           \linewd 0.05}

\begin{document}

\title[Nonlinear resolvents]{Nonlinear resolvents in the unit disk: geometry and dynamics}

\author[M. Elin]{Mark Elin}

\address{Department of Mathematics,
         Ort Braude College,
         Karmiel 21982,
         Israel}

\email{mark$\_$elin@braude.ac.il}

\author[F. Jacobzon]{Fiana Jacobzon}

\address{Department of Mathematics,
         Ort Braude College,
         Karmiel 21982,
         Israel}

\email{fiana@braude.ac.il}


\begin{abstract}
In this paper we present a unified approach to the study of geometric and dynamic properties of nonlinear resolvents of holomorphic generators.

The idea is to apply the distortion theorem we have established. This method allows us to find order of spirallikeness and of strong starlikeness
of resolvents and remove all the restrictions for resolvents to admit quasiconformal extension to the complex plane $\C$.

In addition, we use this method to establish the uniform convergence of the  resolvent family on the whole unit disk  and obtain some characteristics of semigroups generated by these resolvents.

\medskip


\end{abstract}
\maketitle

\section{Introduction}\label{sect-intro}

This paper primarily focuses on nonlinear resolvents, which play a fundamental role in the theory of semigroups of holomorphic mappings.
These semigroups are a natural generalization of semigroups of linear operators. In the one-dimensional case, the work of Berkson and Porta in \cite{B-P}
was a breakthrough in the theory of one-parameter semigroups of holomorphic self-mappings of the open unit disk\footnote{Throughout the paper the term “semigroup” refers to one-parameter semigroups of holomorphic self-maps of the unit disk.}. They proved that every such semigroup is
differentiable with respect to the semigroup parameter, hence, it is generated. Furthermore, they characterized the structure of semigroup generators.
In the multi-dimensional and infinite-dimensional settings, the study of generation theory began with the works \cite{Ab-92} by Abate  and \cite{R-S-96}
by Reich and Shoikhet. Over the past decades, various characterizations of semigroup generators have been found, one of which is mentioned in Theorem~\ref{teorA}
below. See the monographs \cite{B-C-DM-book, E-R-Sbook, E-S-book, R-S1} and references therein for details.

\subsection{Object of study and questions}

One effective way to observe how certain properties of the generator affect the dynamic behavior of the generated semigroup is to use the so-called
exponential formula (see~\eqref{expo-f} below). This formula involves the so-called nonlinear resolvents $G_r$ of the semigroup generator $f$.
These resolvents are defined by the formula $G_r:=(\Id+rf)^{-1}$, $r>0$, see Theorem~\ref{teorA} below. The net $\{G_r\}_{r>0}$ is called the resolvent family.
Leaving aside their importance in dynamic systems, it is worth noting that nonlinear resolvents form a class of biholomorphic self-mappings of the open unit ball,
see~\cite{E-R-Sbook, R-S-96, R-S1}. Thus, they represent a class of mappings inherently interesting from the point of view of geometric function theory.

Surprisingly, the study of geometric properties of nonlinear resolvents has begun only recently in \cite{E-S-S}, where some important facts about resolvents
in the open unit disk were first established. For bounds on coefficient functionals over the class of nonlinear resolvents, see \cite{EJ-coeff20, EJ-est}.
Naturally, obtaining multi-dimensional analogues of these results is a more complicated problem. To date such generalizations are unknown, although there
are partial ones given in \cite{GHK2020, HKK2021}.

In this paper, we elaborate an approach that enables us to establish both geometric and dynamic properties of resolvents by exploring their interrelation.
We specifically deal with resolvents that are  holomorphic in the open unit disk and vanish at zero (by Theorem~\ref{teorA}, it is equivalent to $f(0)=0$),
while the prospect of removing this restriction will be discussed in Section~\ref{sec-discuss}. There are serious reasons for expectations that a similar
approach for the study of the multi-dimensional case will be developed on this basis.

The presentation of the problems to be studied is opened with those important in geometric function theory.

\begin{quest}
Establish distortion and covering results for nonlinear resolvents depending on the resolvent parameter $r>0$.
\end{quest}
To the best of our knowledge, this issue has not been studied yet.
\begin{quest}
Do resolvents admit a quasiconformal extension to $\C$?
\end{quest}
A quasiconformal extension was established in \cite{E-S-S} under the constraint $\left| \arg\frac{f(z)}z \right| < \frac{\pi\alpha}{2}$
for some $\alpha<1$. In Section~\ref{sec-main} below we establish such an extension without any additional restrictions.

To proceed, we recall that starlike and spirallike functions are classical objects in geometric function theory (see Definition~\ref{def-starlike},
the reader can also be referred to the monograph \cite{G-K} and references therein). It was proved in \cite{E-S-S} that any resolvent is a starlike function
of order at least $\frac12\,.$ Therefore, it is natural to raise the problems:
\begin{quest}
Find the sharp order of starlikeness for non-linear resolvents.
\end{quest}

\begin{quest}
Are resolvents spirallike? If the answer is affirmative, find the sharp order of spirallikeness.
\end{quest}

As for dynamic properties, it is known (see formula~\eqref{D-Wp}) that the net of nonlinear resolvents converges to zero uniformly on compact subsets of the open unit disk.
This gives rise to the following questions:
\begin{quest}
Is this convergence uniform on the whole open unit disk? What about convergence of the net of normalized\footnote{An analytic function $f\in \Hol(\D,\C) $ is said to be normalized if $f(0)=f'(0)-1=0$.}  resolvents?
\end{quest}

To explain the importance of the following two questions, recall that, as it was proved in \cite{E-S-S}, every nonlinear resolvent $G_r$, $r>0$, of a generator $f,\ f(0)= 0, $
itself generates a semigroup. Therefore, it is interesting to examine specific characteristics of semigroups generated by resolvents. We will focus on two important properties.
One of them is the rate of convergence of the semigroup. In this connection, we note that although all semigroups (with the exception groups of elliptic automorphisms) converge
to zero uniformly on compact subsets of the disk, some of them tend to zero exponentially uniformly on the whole open unit disk (see Theorem~\ref{thm_kappa} and a detailed explanation in the next section).

Another semigroup property is the possibility of analytic extension with respect to the semigroup parameter into a complex domain. The conditions that ensure
the analytic extension of semigroups along with estimates of the maximal sector in $\C$ to which this extension is possible, are presented in Theorem~\ref{thm-analyt}.

So, the following questions are relevant:

\begin{quest}
  Estimate the rate of convergence of semigroups generated by nonlinear resolvents. Is the convergence uniform on the unit disk?
\end{quest}

\begin{quest}
  Does the semigroup generated by $G_r,\ r>0,$ admit analytic extension with respect to the semigroup parameter  to a sector in the complex plane?
  If the answer is yes, find the maximal angle of  opening of such a sector.
\end{quest}
It seems that this natural question is studied for the first time.

\subsection{Main results}
 In what follows we assume that $f$ is a generator, with $f(0)=0$ and $q=f'(0)\neq0$.

Our approach is to start with a distortion theorem and then successively establish other results.

\begin{theorem}\label{ass1-uiform-conv}
For every $r>\frac2{\Re q}\,$, the  resolvent $G_r$ extends holomorphically  to the disk $D_{\rho(r)}$ of radius $\rho(r)={\displaystyle \left(\sqrt{2r \Re q}-\sqrt{r \Re q -1}\right)^2>1}$, and satisfies
\[
|G_r(z)| \le \sqrt{\frac{2r\Re q}{r\Re q-1}}  -1 \quad \text{ for all } \quad z\in D_{\rho(r)}.
\]
\end{theorem}

Among other things, Theorem~\ref{ass1-uiform-conv} allows us to estimate the image of the function $w\mapsto \displaystyle\frac{wG'_r(w)}{G_r(w)}$.
Using this estimate, we derive the orders of starlikeness, of strong starlikeness and of spirallikeness for nonlinear resolvents (see Definition~\ref{def-starlike}).
More precisely, we show that {\it if $r>\frac6{\Re q}\,,$ then the nonlinear resolvent $G_r$ is
\begin{itemize}
  \item starlike of order $\frac{r\Re q}{r\Re q+6}$;
  \item strongly starlike of order $\frac2\pi \arcsin \frac{6}{r\Re q}\,$;
  \item $\theta$-spirallike of order $\frac{(r\Re q)^2\cos \theta-6r\Re q}{(r\Re q)^2 \cos \theta -36 \cos \theta}$ for any $\theta$ with $|\theta| \leq \arccos \frac{6}{r\Re q}$.
\end{itemize}}

Another interesting consequence of Theorem~\ref{ass1-uiform-conv} concerns the convergence of the net of resolvents. As we have already mentioned,
it is well-known that this net converges to zero uniformly on compact subsets of the open unit disk. We will show subsequently that in fact this convergence is uniform on the whole {\it open} unit disk.

In addition, the net of normalized resolvents $\{(1+rq)G_r\}_{r>0}$ converges to the identity mapping  as $r\to \infty$, uniformly on compact subsets of~$\D$.

\subsection{Outline of the paper}
The outline of the paper is as follows.
In the following section, we recall some notions and provide preliminary results. Sections~\ref{sect-cov-dist}--\ref{sec-semig} contain our results.

The covering and distortion theorems are established in Section~\ref{sect-cov-dist}.
Our approach to the aforementioned questions relies on the distortion theorem and is presented in Section~\ref{sec-main}. In particular, we estimate
the order of starlikeness, strong starlikeness and spirallikeness of nonlinear resolvents. We also use these estimates to study the dynamic behavior of resolvent families.
Section~\ref{sec-semig} is devoted to the study of dynamics of semigroups generated by nonlinear resolvents.

In Section~\ref{sec-discuss}, we make some additional comments and formulate open questions motivated by the results obtained in this work.

\bigskip

\section{Preliminaries}\label{sect-prelim}
\setcounter{equation}{0}

Let $D$ be a domain in the complex plane $\C$.  Denote by $\Hol(D,\C)$ the set of holomorphic functions on $D$, and by $\Hol(D) := \Hol(D,D)$
the set of all holomorphic self-mappings of $D$. In what follows, we denote by $D_r$ the open disk of radius $r$, namely, $D_r:=\left\{z:\ |z|<r\right\}$.
Correspondingly, $\D=D_1$ is the open unit disk.

Let $\Omega$ be the subclass of $\Hol(\D)$ consisting of functions vanishing at the origin:
\begin{equation}\label{def-U}
\Omega=\{ \omega \in \Hol(\D):\ \omega(0)=0 \}.
\end{equation}
The identity mapping on $\D$ will be denoted by $\Id$.

To define nonlinear resolvents, the main object of the paper, recall that a mapping $f\in\Hol(\D,\C)$ is called an {\it (infinitesimal) generator} if for every $z\in\D$, the Cauchy problem
\begin{equation}  \label{nS1}
\left\{
\begin{array}{l}
\frac{\partial u(t,z)}{\partial t}+f(u(t,z))=0    ,     \vspace{2mm} \\
u(0,z)=z,%
\end{array}%
\right.
\end{equation}%
has a unique solution $u=u(t,z)\in\D$ defined for all $t\geq 0$. In this case, the unique solution of \eqref{nS1} forms a semigroup generated by $f$; see, for example, \cite{B-C-DM-book, E-R-Sbook, E-S-book, R-S1, SD}. It turns out that generators can be characterized as follows.

 \begin{theorem}[see \cite{R-S1, SD, E-R-Sbook, E-S-book} for details] \label{teorA}
 Let $f\in \Hol(\D , \C),\  f\not\equiv0$. Then $f$ is a generator on $\D$ if and only if it satisfies the so-called range condition:
\[
\left(\Id +rf\right)(\D)\supset\D\qquad    \mbox{for all }\quad r>0,
\]
and $G_r:=(\Id+rf)^{-1}$ is a well-defined self-mapping of $\D$.
 \end{theorem}

The mappings $G_r\in\Hol(\D),\ r>0,$ are called the {\it nonlinear resolvents} of the generator $f$, the net $\{G_r\}_{r>0}$ is the {\it resolvent family} for $f$. These are the main objects of the study in this paper.

Numerous properties of nonlinear resolvents can be found in the books \cite{SD, R-S1, E-R-Sbook}. In particular, the solution of the Cauchy problem~\eqref{nS1} can be reproduced by the following exponential formula:
\begin{equation}\label{expo-f}
u(t,\cdot) = \lim_{n\rightarrow \infty } \left(G_{\frac{t}{n}}\right)^{[n]},
\end{equation}
where $G^{[n]}$ denotes the $n$-th iterate of a self-mapping $G$ and the limit exists in the topology of uniform convergence on compact subsets of $\D.$

It is known that a generator has at most one null point in $\D$. This null point $\tau\in\D$ is known to be the \textit{Denjoy--Wolff point} for the semigroup
$\left\{u(t,\cdot)\right\} _{t\geq 0}$ defined by \eqref{nS1} as well as for the resolvent family $\{G_r\}_{r>0}$. More precisely, if the semigroup does not contain an elliptic automorphism, then
\begin{equation}\label{D-Wp}
\tau= \lim\limits_{t\rightarrow \infty } u(t,z) =\lim_{r\to\infty}G_{r}(z) \quad \mbox{for any}\quad  z\in\D.
\end{equation}
Moreover, the convergence in \eqref{D-Wp} is uniform on compact subsets of the open unit disk.

\vspace{2mm}

As is noted, we concentrate on the  case $\tau=0$. In this case the famous  Berkson--Porta representation formula for infinitesimal generators (see \cite{B-P}) becomes
\begin{equation}\label{B-P-repres}
f(z)=zp(z) \quad \text{with} \quad \Re p(z)\ge 0 \quad (z\in \D)
\end{equation}
and $\lim\limits_{r\to\infty} G_r(z)=0,$ uniformly on compact subsets of $\D$, by~\eqref{D-Wp}. There is a simple verifiable condition for
the convergence of the semigroup to its Denjoy--Wolff point $\tau=0$ to be uniform on the {\it whole disk}~$\D$:

\begin{theorem}[see \cite{FSG, E-S-S, E-R-Sbook}]\label{thm_kappa}
Let $\kappa>0$ be a constant. The semigroup $\{u(t,\cdot)\}_{t\ge0}$ generated by~$f$, $f(z)=zp(z),$ satisfies the estimate $|u(t,z)|\le |z|e^{-\kappa t }$ for all $t>0$ and $z\in\D$ (and consequently $u(t,z)\to0$ as $t\to\infty$ uniformly on $\D$)  if and only if $\Re p(z)\ge\kappa,\ z\in\D$.
\end{theorem}
If a semigroup satisfies the estimate $|u(t,z)|\le |z|e^{-\kappa t }$ with some $\kappa>0$, it is called {\it exponentially squeezing}. The number $\kappa$ is called {\it squeezing ratio}.

To present another property of semigroups, which will appear below, we recall that semigroups by definition are well-defined for real non-negative values of the parameter $t$ only. However, it may happen that for every fixed $z\in\D$, the function $u(\cdot,z)$ can be analytically extended to some sector in the complex plane. Analyticity of semigroups was recently studied in \cite{A-C-P, E-J-17, E-S-Ta}, see also \cite[Chapter 6]{E-R-Sbook}. The following fact will be used in the sequel.

\begin{theorem}\label{thm-analyt}
  Let $\alpha,\beta\in(0,\frac\pi2)$. The semigroup $\{u(t,\cdot)\}_{t\ge0}$ generated by~$f$, $f(z)=zp(z),$ can be analytically extended to the sector $\{t\in\C: \ \arg t\in(-\alpha,\beta)\}$ for all $z\in\D$ if and only if $-\frac\pi2+\alpha<\arg p(z)<\frac\pi2-\beta,\ z\in\D$.
\end{theorem}
In this connection we also notice that due to the exponential formula~\eqref{expo-f}, the analyticity of a semigroup in some sector follows from the analyticity of all the resolvents in the same sector; see \cite[Section~6.2]{E-R-Sbook}.

 \vspace{2mm}

To proceed, recall several important classes of univalent functions intensively studied in geometric function theory.

\begin{defin}\label{def-starlike}
  Let $h \in \Hol(\D,\C),\ h(0)=0$ and  $h'(0)\neq 0$. We say that 
  \begin{itemize}
\item [] 
$h$ is starlike of order $\alpha\in(0,1) $ if $\Re\left(\frac{zh'(z)}{h(z)}\right)>\alpha$ for all $z\in\D$;
\item []
$h$ is  $\theta$-spirallike of order $\alpha\in(0,1) $ if $\Re \left(e^{-i\theta}\frac{zh'(z)}{h(z)}\right)>\alpha \cos \theta$ for all $z\in\D$;
\item [] 
$h$ is strongly starlike of order $\beta\in(0,1)$ if $\left|\arg\frac{zh'(z)}{h(z)}\right|<\frac{\pi\beta}{2}$ for all $z\in\D.$
  \end{itemize}
\end{defin}
Concerning the classes of starlike and spirallike functions, the  reader can be referred to the books \cite{Good, G-K}. Specifically we will need the following fact which can be obtained using the Riesz--Herglotz formula.

\begin{propo}[Problem 4, p. 172 in \cite{Good}]\label{prop-present}
Let $h \in \Hol(\D,\C),\ h(0)=0$ and  $h'(0)\neq 0$.  Then the function $h$ is starlike of order $\alpha$ if and only if it admits the integral representation
\[
h(z)=z \exp\left[ -2(1-\alpha)\oint_{\partial\D} \log\left( 1-z\overline{\zeta}\right)d\mu(\zeta) \right]
\]
  with some probability measure $\mu$ on the unit circle.
\end{propo}

\begin{defin}\label{def-starlike}
  Let $h\in \Hol(\D),\ h(0)=0,$ be a univalent function. One says that $h$ is hyperbolically convex if for all points $a, b \in h(\D)$, the arc of the hyperbolic geodesic in $\D$ connecting these points lies in $h(\D)$.
\end{defin}
More details on hyperbolic geodesics can be found \cite[Section 1.3]{B-C-DM-book}. As to hyperbolically convex functions, see, for example, \cite{Ma-Mi-94, Me-Mi-91}. It was proved in \cite{Me-Pomm} that every hyperbolically convex function is starlike of order $\frac12$\,.

\bigskip

\section{Covering and distortion results}\label{sect-cov-dist}
\setcounter{equation}{0}

The purpose of this section is to establish covering and distortion theorems for families of nonlinear resolvents. We start with a more general situation.

Let $\alpha, \beta \in\C$ with $\Re \alpha \overline{\beta} >0.$ Consider the class $\A_{\alpha,\beta}$ consisting of functions holomorphic in the open unit disk $\D$ and satisfying the conditions
  \begin{equation}\label{ineq}
 \A_{\alpha,\beta}=\!\left\{ F: F(0)=F'(0)-\beta=0,\,  \Re \frac1\alpha\left(\frac{F(z)}{ z}-\beta\right) >-\frac 1 2\!\right\}\!.
  \end{equation}
The inequality in~\eqref{ineq} is equivalent to $\displaystyle\frac1\alpha\left(\frac{F(z)}{ z}-\beta\right) \prec \frac{z}{1-z}, $ where the subordination relation means that there exists a function $\omega \in \Omega$ such that
$\displaystyle\frac1\alpha\left(\frac{F(z)}{ z}-\beta\right) = \frac{\omega(z)}{1-\omega(z)}.$
Define $
\displaystyle\psi(z)=\beta+\frac{\alpha z}{1-z}\,.
$
Then
\begin{equation}\label{classA}
\A_{\alpha,\beta}:=\left\{F \in \Hol(\D,\C): \frac{F(z)}{z}\prec\psi\right\}.
\end{equation}
Clearly, every $F \in \A_{\alpha,\beta}$ is locally univalent at the origin and the inverse function $F^{-1}$ satisfies $F^{-1}(0)=0$. Consider the set of inverse functions
\begin{equation}\label{classB}
\BB_{\alpha,\beta}:=\left\{F^{-1} : F\in \A_{\alpha,\beta}\right\}.
\end{equation}
We now establish the radius of univalence for  the class $\BB_{\alpha,\beta}$ as well as covering and distortion results.

     \begin{theorem}\label{th_posi1}
For $\alpha, \beta \in\C$ with $\Re \alpha \overline{\beta} >0, $   denote $M=1-\Re\frac{\beta}{\alpha}\,$.
   Every function $G\in\BB_{\alpha,\beta}$ is univalent  in the disk $D_R$, where
 \[
  R= \left\{ \begin{array}{lc}
                |\alpha|\left(\frac{1}{2}-M\right) , & \  \mbox{if } \Re \frac{\beta}{\alpha}>\frac34\,, \vspace{2mm} \\
                  |\alpha|\left(1-\sqrt{M}\right)^2  \,, &  \ \mbox{if }   \Re \frac{\beta}{\alpha} \le \frac34\,,
                \end{array}
 \right.
 \]
and satisfies $D_{ R_1} \supset G(D_{R}) \supset D_{ R_2}$ with
\[
 R_1= \left\{ \begin{array}{ccc}
                 1, & \quad \mbox{if } \Re \frac{\beta}{\alpha}>\frac34\,, \vspace{2mm}\\
                  \frac{1}{\sqrt{M}}-1 , &  \quad\mbox{if } \Re \frac{\beta}{\alpha}\le\frac34\,,
                \end{array}
 \right. \qquad   R_2=\displaystyle\frac{RR_1}{R_1|\beta| + \sqrt{R_1^2|\beta|^2-R^2}}\,.
 \]
  \end{theorem}

 \begin{proof} Let us represent the number $\frac{\alpha}{2\beta}$ in the form $se^{-i\theta}$, or $\alpha=2\beta se^{-i\theta}$, 
 where $|\theta|<\frac\pi2$ since ${\Re\alpha\overline{\beta}>0}$. Then by \eqref{ineq}, we  get  $ \Re e^{i\theta}\frac{F(z)}{\beta z} >\cos\theta-s.$ 
 Therefore, the embedding  $ G(D_{ R}) \subset D_{ R_1}$ follows from \cite[Corollary 3.2]{E-S-2020a}. Moreover, the proof of Theorem~3.1 
 in \cite{E-S-2020a} yields that the function $G$ is univalent.

 It remains to prove the covering relation $D_{R_2} \subset G(D_{R})$.
 It follows from Corollary 3.2 in  \cite{E-S-2020a} that $G$ maps $D_R$ onto a hyperbolically convex subdomain of the disk $D_{R_1}$. Therefore, the function $h$ defined by  $h(z)={G(Rz)}/{R_1}$
  belongs to $\Hol(\D)$ and is hyperbolically convex. By the result in \cite{Me-Mi-91} (see also \cite[Theorem~2]{Ma-Mi-94}), the image $h(\D)$ contains the disk of radius $\frac{|h'(0)|}{1+\sqrt{1-|h'(0)|^2}}\,$. Since $h'(0)=\frac{G'(0)R}{R_1}=\frac{R}{\beta R_1}\,,$ one concludes that $G(D_R)$ contains the disk of radius
 \[
    \frac{|h'(0)|R_1}{1+\sqrt{1-|h'(0)|^2}} = \frac{RR_1}{R_1|\beta| + \sqrt{R_1^2|\beta|^2-R^2}} = R_2.
 \]
 The proof is complete.
 \end{proof}

 \begin{remar}
  In the proof of this theorem, the hyperbolic convexity was used only for the purpose of proving the covering result. 
  It is worth mentioning that certain covering result can be obtained in the absence of this property.
 Indeed, let $w\not\in G(D_{R}) $. Then the function $h$ defined by $h(z)=\displaystyle\frac{G(Rz)}{w-G(Rz)}$ is holomorphic and univalent in the unit disk. It can easily be seen that
   \[
   h(0)=0,\quad h'(0)=\frac{G'(0)R}w\quad\mbox{and}\quad h''(0)=\frac{G''(0)w + 2G'(0)^2}{ w^2}R^2.
   \]
   Since by the famous Bieberbach theorem $|h''(0)|\le 4|h'(0)|$ (see, for example, \cite{G-K}), we have
   \[
   4|G'(0)w|\ge R\left| G''(0)w+2G'(0)^2\right| \ge 2R|G'(0)|^2-R|G''(0)w|.
   \]
     It follows from \cite[Proposition~4.1]{E-J-21a} that $G'(0)=\frac1\beta$ and $|G''(0)|\le \frac{|\alpha|}{|\beta|^3}$. 
     This leads to $|w|\ge \displaystyle\frac{|\beta|R}{|\beta|^2 +|\Re \beta|R}$. Thus, $G(D_R)$ covers the disk of radius $\displaystyle\frac{|\beta|R}{|\beta|^2 +|\Re \beta|R}.$
 \end{remar}

\begin{examp}\label{ex-starlike}
  Consider the set of all functions $F\in\Hol(\D,\C)$ such that $F(0)=F'(0)-1=0$ and $\Re\frac{F(z)}{z}\ge\frac12\,.$ This is equivalent to  $F\in \A_{1,1}$, that is, to the choice  $\alpha=\beta=1\,.$   Theorem~\ref{th_posi1} implies that every $G\in\BB_{1,1}$ is univalent in the disk of radius $\frac{1}{2}$ and  $D_{\frac{1}{2+\sqrt{3}}}\subset G(D_{\frac{1}{2}})\subset\D.$
\end{examp}

We now apply Theorem~\ref{th_posi1} to  the special case where function $\psi$ maps the open unit disk onto the half-plane ${\{w:\,\Re w>1\}}$.
Keeping in mind our interest in resolvents, we fix $q\in\C$ with $\Re q>0$ and choose $\beta=1+rq$ and $\alpha=2r \Re q$. Consider the net $\{\psi_r\}_{r>0}$ such that
\begin{equation}\label{psi-r}
\psi_r(z)=1+r\,\frac{q+\overline{q}z}{1-z}=1+rq+2r\Re q\sum_{n=1}^\infty z^n.
\end{equation}
Accordingly, $\A_r:=\{F \in \Hol(\D,\C): \frac{F(z)}{z}\prec\psi_r\}$, cf. \eqref{classA}.
We then formulate criteria for a holomorphic function $F$ to belong to the class~$\A_r$.
\begin{lemma}
  Let $F\in\Hol(\D,\C),\ F(0)=0.$ Then the following conditions are equivalent:
  \begin{itemize}
  \item [(i)] $F\in \A_r$;
  \item [(ii)] $\Re \frac{F(z)}{z}\ge1$ for all $z\in\D$ and $F'(0)=1+rq$;
  \item [(iii)] $F(z)=z+rz\frac{q+\overline{q}\omega(z)}{1-\omega(z)}$ for some $\omega\in\Omega$;
 \item [(iv)] the function $f$ defined by $f(z)=\frac{F(z)-z}r$ is a generator on $\D$.
\end{itemize}
\end{lemma}
This lemma follows directly from our notations and formula~\eqref{B-P-repres}. If it is the case, assertions (iii) and (iv) immediately imply that
\begin{equation} \label{G*-1}
 f(z)=z\frac{q+\overline{q}\omega(z)}{1-\omega(z)}\,,\quad  \omega\in\Omega.
\end{equation}
Hence the (right) inverse function ${F^{-1}\!=:G(=G_r)}$, which, in fact, solves the functional equation
\begin{equation} \label{G*-2}
G_r+ rf\circ G_r=\Id,
\end{equation}
is holomorphic in the open unit disk $\D$ by Theorem~\ref{teorA}. (Recall that $G_r$ is called the resolvent of $f$. It is a univalent self-mapping of $\D$, see Section~\ref{sect-prelim}.)

In what follows, we focus on the family
$$\JJ_r:=\BB_{2r\Re q,1+rq},$$
cf. \eqref{classB}, consisting of the resolvents $G_r=(\Id + rf)^{-1}$ with a fixed $q=f'(0)$.

\vspace{2mm}

 As we have already mentioned, the resolvent family converges to  the null point of $f$  as $r \to \infty$, uniformly on compact subsets of the unit disk. 
 We will now show that for $r>\frac{2}{\Re q}$, nonlinear resolvents extend holomorphically to a disk of prescribed radius,   and prove the distortion and covering results. This enables us to establish that $G_r$ tends to $0$ as $r\to\infty$ uniformly on $\D$.

 \begin{theorem}\label{th_resol1}
Let $r>\frac2{\Re q}\,.$ Every element $G_r$ of $\JJ_r$ can be extended as a univalent function to the disk $D_{\rho(r)},\  \rho(r)={\displaystyle \left(\sqrt{2r \Re q}-\sqrt{r \Re q -1}\right)^2>1}$, and satisfies
\begin{itemize}
  \item [(a)] $D_{ \rho_1(r)}\supset G_r(D_{ \rho(r) }) \supset D_{\rho_2(r)}$ with $$\rho_1(r)= \displaystyle\sqrt{\frac{2r\Re q}{r\Re q-1}}  -1,\ \quad  \rho_2(r)=\displaystyle\frac{\left(\sqrt{2r \Re q}-\sqrt{r \Re q -1}\right)^2} {|1+rq| + \sqrt{2+r\Re q+r^2|q|^2}}\,;$$
  \item [(b)] $G_r(\D)\subset D_{\rho_3(r)},$ where $\displaystyle\rho_3(r)=\frac{3}{1+r\Re q}$.
\end{itemize}
Furthermore, $G_r(\D)\supset D_{\rho_4(r)},$ where $\displaystyle \rho_4(r)=\frac1{|1+rq| + \sqrt{|1+rq|^2 -1}}\,$ for any $r>0.$
\end{theorem}

\begin{proof}
  Since $\beta= 1+rq$ and $\alpha= 2r\Re q$, the condition $r>\frac{2}{\Re q}$ is equivalent to $\Re\frac\beta\alpha<\frac34\,$. 
  Following the notation in Theorem~\ref{th_posi1}, $M=1-\Re\frac{\beta}{\alpha}=\frac{r\Re q-1}{2r\Re q}>\frac14\,.$ We now substitute 
  these expressions for those in Theorem~\ref{th_posi1} and then arrive at the univalence of $G_r$ in $D_{\rho(r)}$ as well as the inclusions in part (a).

 In fact, the proved distortion result implies part (b). Indeed, consider the function $h$ defined by $h(z):=\frac{G_r(\rho(r)z)}{\rho_1(r)}\,$. 
 By part (a),  $h$ is a self-mapping of the open unit disk and hence by the Schwarz lemma, we have $|h(z)|\le|z|$, or equivalently, 
 $|G_r(\rho(r)z)|\le \rho_1(r)|z|$ for all $z\in\D$. Denote $\zeta=\rho(r)z$. Then $|G_r(\zeta)|\le \frac{\rho_1(r)|\zeta|}{\rho(r)}$ for $\zeta\in D_{\rho(r)}.$ 
 Since $\rho(r)>1$, one can take, in particular, $\zeta\in\D$. This yields
\[
|G_r(\zeta)|\le \frac{\rho_2(r)}{\rho(r)} = \frac{1}{(r\Re q-1)\left( \sqrt{\frac{2r\Re q}{r\Re q-1}} -1 \right)} \le\frac{3}{1+r\Re q}\,,
\]
which proves part (b).

  To complete the proof, we recall that according to a result in \cite{E-S-S}, every resolvent is a hyperbolically convex function. 
  In addition, it follows from \cite {Me-Mi-91} (see also \cite[Theorem 2]{Ma-Mi-94}) that the image of the unit disk under every 
  hyperbolically convex function $h$ normalized by $h(0)=0$,  $h'(0)=\delta\in(0,1),$ contains the disk of radius $\frac\delta{1+\sqrt{1-\delta^2}}$. 
  Since $G_r'(0)=\frac1{1+rq}\,$, we obtain \\  $G_r(\D)\supset D_{\rho_4(r)}$, which completes the proof.
\end{proof}

Part (b) of the above theorem immediately implies the following fact.
\begin{corol}\label{corr-boundary-fix}
The net $\{G_r\}_{r>0}$ converges to zero uniformly on $\D$ as~${r\to\infty}$.
\end{corol}

In addition, it was shown in \cite{E-S-S} that a point $\zeta \in \partial \D $
is a boundary regular fixed point of $G_{r}$, the resolvent of $f$, if and only if it is a boundary regular null point of  $f$ and $r<1/|f^{\prime }(\zeta )|<+\infty$.
At the same time, it follows from \cite{E-S-Ta11} that $1/|f^{\prime }(\zeta )|\leq {2}/{\Re q}$. Consequently, if $r>{2}/{\Re q}$, 
then $G_r$ has no boundary fixed points. The last conclusion follows directly from part (b) of Theorem~\ref{th_resol1} as well.

\bigskip

\section{Order of starlikness and spirallikeness}\label{sec-main}
\setcounter{equation}{0}

In this section, we present our approach to obtaining geometric and dynamic properties of nonlinear resolvents.
For this, the results proven above will intensively be used. We first describe the range of the function $\frac{wG'_r(w)}{G_r(w)}$. 
This enables us to establish the order of starlikeness, order of spirallikeness and order of strong starlikeness (see Definition~\ref{def-starlike}) 
of the resolvent $G_r$ as a function depending on the resolvent parameter $r$.

To this end, we define the function
 \begin{equation}\label{Ar}
 A(r):=\displaystyle\frac{6r (1+r )}{(1+r)^3- 3(5r-1) }
 \end{equation}
 and denote the largest real root of the equation $A(r)=1$ by $r_0$. It can be calculated that $r_0 = 1+2\sqrt{7}\cos\left(\frac{1}{3}\arctan\frac{3\sqrt{31}}{8}\right) \approx 5.92434$.

\begin{theorem}\label{th-r-star}
Let $G_r\in \JJ_r$, where $r\Re q> r_0$. Then for all $w\in\D,$
\begin{equation*}
\left|\frac{wG'_r(w)}{G_r(w)} - \frac{1}{1-A^2(r\Re q)}\right| \le \frac{A(r\Re q)}{1-A^2(r\Re q)}\,.
\end{equation*}
\end{theorem}

\begin{proof}
Recall that $G_r=(\Id +rf)^{-1}$, where $f(z)=zp(z)$ with $p(z)=\frac{q+\overline{q}\omega(z)}{1-\omega(z)}\,$, $\omega \in \Omega$, see \eqref{G*-1}. Hence, $\Re p(z)>0$ for all $z\in \D$.
Further, formula \eqref{G*-2} implies $\frac{w}{G_r(w)}=1+rp(G_r(w))$. Thus differentiating \eqref{G*-2} we get
\begin{equation}\label{starGr}
\frac{wG'_r(w)}{G_r(w)}=\frac{1+rp\circ G_r(w)}{1+r\left(p\circ G_r(w)+ G_r(w) \cdot p' \circ G_r(w)\right)}\,.
\end{equation}
According to assertion (b) of Theorem~\ref{th_resol1}, the inequality $|G_r(w)|\leq\frac{3}{1+r\Re q}$ holds for all $w\in \D$. 
Hence, our aim is to find the range of  $\frac{1+rp(z)}{1+r\left(p(z)+ z  p' (z)\right)}$ whenever $|z|\leq \frac{3}{1+r\Re q}$.

The Riesz--Herglotz formula gives
\begin{equation*}
p(z)=\int_{|\zeta|=1} \frac{1+z\overline{\zeta}}{1-z\overline{\zeta}}\,d\mu(\zeta)+i\gamma,
\end{equation*}
for some non-negative measure $\mu$ on the unit circle and a number $\gamma \in \R,$ so that
\[
p(0)=q=\int_{|\zeta|=1}d\mu(\zeta)+i\gamma.
\]
Denote also
\begin{eqnarray*}\label{ineqA1}
A_r(z,\zeta)&:=&\frac{2 r\Re q|z\overline{\zeta}|}{1+r\Re q- 2\Re z\overline{\zeta} + |z\overline{\zeta}|^2 (1-r\Re q) }\,,
\end{eqnarray*}
\begin{eqnarray*}\label{ineqA1}
B_r(z)&:=&\Re (1+rp(z)) \qquad \text{and} \qquad C_r(z):=\left| rzp'(z) \right|.
\end{eqnarray*}
One can see that for all $z$ with $|z|<\frac{3}{1+r\Re q}$ and $\zeta$ with $|\zeta|=1$,
\begin{equation*}\label{ineqAr}
 A_r(z,\zeta)\leq A_r\left(\frac3{1+r\Re q},1 \right) =  A(r\Re q).
\end{equation*}
Also,
\begin{equation*}
B_r(z)=\int_{|\zeta|=1} \frac{1+r\Re q- 2\Re z\overline{\zeta} + |z\overline{\zeta}|^2 (1-r\Re q) }{\Re q |1-z\overline{\zeta}|^2}d\mu(\zeta),
\end{equation*}
so that
\begin{eqnarray*}
C_r(z) &\leq& \int_{|\zeta|=1} \frac{2 r|z\overline{\zeta}|}{|1-z\overline{\zeta}|^2}d\mu(\zeta) \\
\nonumber   &=& \int_{|\zeta|=1}  A_r(z,\zeta)\frac{1+r\Re q- 2\Re z\overline{\zeta} + |z\overline{\zeta}|^2 (1-r\Re q) }{\Re q |1-z\overline{\zeta}|^2}d\mu(\zeta)\\
\nonumber  &\leq& A(r\Re q)B_r(z).
\end{eqnarray*}
Thus,
\begin{equation*}
 \left|  \frac{rzp'(z)} {1+rp(z)}\right| \leq \frac{C_r(z)} {B_r(z)}\leq  A(r\Re q).
\end{equation*}

Therefore, the function $\displaystyle\frac{1+r\left(p(z)+ z  p' (z)\right)}{1+rp(z)}$ takes values in the disk centered at $1$ and of radius $A(r\Re q)$. 
A straightforward calculation based on formula~\eqref{starGr} shows that all of the values of $\displaystyle\frac{wG'_r(w)}{G_r(w)}$ are located in the disk 
centered at $\displaystyle \frac{1}{1-A^2(r\Re q)}$ and of radius $\displaystyle \frac{A(r\Re q)}{1-A^2(r\Re q)}\,.$ The proof is complete.
\end{proof}

Bearing in mind that Definition~\ref{def-starlike} of order of starllikeness and spirallikeness involves the range of the function $\frac{wG'_r(w)}{G_r(w)}$ described in this theorem, we  deduce the following geometric conclusion.
\begin{corol}\label{corr-order-sp-st}
Let $G_r\in \JJ_r$, where $r\Re q> r_0$ and $\theta \in \R$ with $|\theta| \leq \arccos \frac{6}{r\Re q}$.
Then $G_r$ is a $\theta$-spirallike function of order $$\displaystyle \alpha_{r,\theta}:=\frac{\cos \theta -A(r\Re q)}{(1-A^2(r\Re q))\cdot \cos \theta}\,.$$
Consequently, $G_r$ is starlike of order $\displaystyle \alpha_r:=\frac1{1+A(r\Re q)}$ and strongly starlike of order $\beta_r:=\displaystyle\frac2\pi \arcsin A(r\Re q)$.
\end{corol}

It was shown in \cite{E-S-S} that $G_r$ is a starlike function of order $\frac12$ for each $r>0$. Note that if $r \Re q > r_0$, then $A(r \Re q)<1$ and $\lim\limits_{r\to \infty} A(r\Re q)=0$.
Thus, $\alpha_r > \frac{1}{2}$ and $\lim\limits_{r\to \infty} \alpha_r = 1$.
Therefore, Corollary~\ref{corr-order-sp-st} considerably improves the mentioned result. Moreover, we immediately get
\begin{corol}\label{cor-tends_to_z}
The net of functions  $\left\{(1+rq)G_r(z)\right\}$ converges to $z$  as $r\to \infty$, uniformly on compact subsets of the unit disk.
\end{corol}

\begin{remar}\label{rem-1}
One can see that if $r\Re q > 6$, then $A(r) < \frac{6}{6-r_0+r}< \frac6r$. This gives the bounds for the orders of spirallikeness, starlikeness and strong starlikeness. Namely,
\begin{itemize}
  \item order of $\theta$-spirallikeness $\alpha_{r,\theta}$ is greater than  $\frac{r\Re q(r\Re q\cos \theta-6)}{((r\Re q)^2  -36)\cos \theta }$;
  \item order of starlikeness $\alpha_r$ is greater than  $\frac{6-r_0+r\Re q}{12-r_0+r\Re q}>\frac{r\Re q }{6+r \Re q }\,;$
  \item order of strong starlikeness $\beta_r$ is less than $\frac2\pi\arcsin\frac6{6-r_0+r\Re q}< \frac2\pi \arcsin \frac{6}{r\Re q}\,.$
\end{itemize}
\end{remar}

It was proved in \cite{F-K-Z} (see also \cite{Su-12}) that any strongly starlike function of order $\alpha$ extends to a $\sin(\pi\alpha/2)$-quasiconformal automorphism of $\C$. 
Therefore, Corollary~\ref{corr-order-sp-st} entails
\begin{corol}\label{cor-quasi}
  Any function $G_r\in\JJ_r,\ r\Re q> r_0$, can be extended to a $k$-quasiconformal  automorphism of \,$\C$ with  $k\bigl(=k(r)\bigr)=A(r\Re q)$.
\end{corol}

Note also that under the additional condition that a generator $f$ satisfies $\left| \arg\frac{f(z)}z \right| < \frac{\pi\alpha}{2}$, it was proved in \cite{E-S-S} that all its  resolvents admit a $(\sin\pi\alpha)$-quasiconformal extension to $\C$. Corollary \ref{cor-quasi} provides quasiconformal extension without additional conditions.
Moreover, $\lim\limits_{r \to \infty}k(r)=0$.



\bigskip

\section{Semigroups generated by resolvents}\label{sec-semig}
\setcounter{equation}{0}

As we have already mentioned in the introduction, every resolvent $G_r$ generates a semigroup, see \cite{E-S-S}. 
Therefore, it is natural to study the properties of semigroups generated by nonlinear resolvents. Our first observation is straightforward.

\begin{propo}\label{prop-bound-fix}
Let $r>0$. Then the semigroup $\{u(t,\cdot)\}_{t\ge0}$ generated by $G_r$ has no boundary regular fixed point\footnote{ It was noted by the anonymous reviewer that our proof and \cite [Proposition 13.6.1]{B-C-DM-book} imply that  $\{u(t,\cdot)\}_{t\ge0}$ does not have {\it any} boundary fixed point.}.
\end{propo}
\begin{proof}
Indeed, $G_r$ is a holomorphic function that satisfies $G_r(0)=0.$ Therefore, there is $\varepsilon>0$ such that the image $G_r(D_{0.2})$ covers the disk $D_\varepsilon$. 
Since $G_r$ is univalent in $D$, we have $|G_r(z)|>\varepsilon$ whenever $0.8<|z|<1$. Consequently, $G_r$ has no boundary null point. 
Because each boundary regular fixed point of a semigroup should be boundary regular null point for its generator (see, for, example, \cite{E-S-book}), the  conclusion follows.
\end{proof}

In the next theorem we establish the uniform convergence (on the whole disk $\D$) of such semigroups as well as their analyticity in a sector  with respect to the parameter $t$.

\begin{theorem}\label{thm-estim}
  Let $G_r\in\JJ_r$ with $r\ge \frac6{\Re q}\,$. Denote $\gamma_r:=\frac{1-A(r\Re q)}{1+A(r\Re q)}$, where function $A$ is denoted by \eqref{Ar}. Then for the semigroup $\{u(t,\cdot)\}_{t\ge0}$ generated by $G_r$ the following assertions hold:
  \begin{itemize}
    \item [(i)]  $\{u(t,\cdot)\}_{t\ge0}$ is exponentially squeezing with squeezing ratio $\kappa(r) :=\displaystyle\frac{\left( \Re(1+rq)^\frac1{\gamma_r} \right)^{\gamma_r}} {2^{1-\gamma_r}|1+rq|^2}\,$ (so converges to $0$ as $r\to\infty$, uniformly on $\D $);
    \item [(ii)]  for every $z\in\D$, $u(\cdot,z)$ can be analytically extended to the sector $$\left\{t\in\C: \left|\arg t - \arg(1+rq)\right| < \frac{ \pi\gamma_r}{ 2}\right\}.$$
    \end{itemize}
  \end{theorem}

\begin{proof}
Since the function $G_r$ is starlike of order $\alpha_r$ by Theorem~\ref{th-r-star}, the function $(1+rq)G_r$ is a normalized starlike function of the same order,  and hence by Proposition~\ref{prop-present} it admits the integral representation
\[
(1+rq)G_r(z)=z \exp\left[ -2(1-\alpha_r)\oint_{\partial\D} \log\left( 1-z\overline{\zeta}\right)d\mu_r(\zeta) \right]
\]
  with some probability measure $\mu_r$ on the unit circle.
  Therefore, by the same Proposition~\ref{prop-present}, the function $z\left(\frac{(1+rq)G_r(z)}z\right)^{\frac1{2(1-\alpha_r)}} $  is starlike of order~$\frac12$\,. Since $2(1-\alpha_r)=1-\gamma_r$, the Marx--Strohh\"acker theorem (see, for example, \cite[Theorem~2.6a]{M-M}) implies that \begin{equation}\label{aux1}
  \Re \left(\frac{(1+rq)G_r(z)}z\right)^{\frac1{1-\gamma_r}}>\frac12\,.
  \end{equation}

According to Theorem~\ref{thm_kappa}, to prove assertion (i), we have to show that $ \Re \frac{G_r(z)}{z} > \kappa(r).$
Let us denote for short $w=\frac{G_r(z)}z$ and $B(r)=1-\gamma_r$.
Our aim is to minimize $\Re w$ under the condition $\Re \left((1+rq)w\right)^{\frac1{B(r)}} = \frac12\,,$ see \eqref{aux1}. 
In other words, we have to minimize the function $\zeta(t):=\Re\frac{\left(\frac12+it\right)^{B(r)}}{1+rq}\,.$ 
Equating $\zeta'(t)$ to zero, we get $\arg \left(\frac12+it\right)^{1-B(r)} = \arg(1+r\overline q)$ at the minimal point of $\zeta$, 
or equivalently, $\frac12+it = M(1+r\overline q)^{\frac1{1-B(r)}}$ for some $M>0.$ This leads to $\frac12+it =\frac{(1+r\overline q)^{\frac1{1-B(r)}}}{2\Re (1+rq)^{\frac1{1-B(r)}}}\,. $ Thus,
\begin{eqnarray*}
  \min_{t\in\R} \zeta(t) &=& \Re\frac{(1+r\overline q)^{\frac{B(r)}{1-B(r)}}} {(1+rq)2^{B(r)} \left(\Re(1+rq)^{\frac1{1-B(r)}} \right)^{B(r)}} \\
    &=&  \frac{\left( \Re(1+rq)^\frac1{1-B(r)} \right)^{1-B(r)}} {2^{B(r)}|1+rq|^2}=\kappa(r),
\end{eqnarray*}
and we are done.

To prove assertion (ii), we use the same notations. We now estimate the values of $\arg w$, or, in other words, the values of the function 
$\xi(t):=\arg\frac{\left(\frac12+it\right)^{B(r)}}{1+rq} = B(r)\arg\left(\frac12+it\right) - \arg(1+rq).$ Obviously, 
$$\xi(t)\in \left(-\frac\pi2 B(r)-\arg(1+rq),\frac\pi2 B(r)-\arg(1+rq) \right).$$
Applying Theorem~\ref{thm-analyt}, we complete the proof.
 \end{proof}

It is easy to see that the squeezing ratio $\kappa(r)$ presented in Theorem~\ref{thm-estim} satisfies the condition $|1+rq|\kappa(r)\to \frac{\Re q}{|q|}$ as $r\to\infty$. 
Moreover, in the case of real $q$  we have $\kappa(r) =\displaystyle\frac{1} {2^{1-\gamma_r}(1+rq)}\,$. Since $\gamma_r\to1$ as $r\to\infty,$ we conclude that assertion (i) 
of Theorem~\ref{thm-estim} is a strong improvement of the result in \cite{E-S-S}, which states that the semigroup, generated by $G_r$ converges to zero uniformly on~$\D$ 
with squeezing ratio $\kappa (r)=1/[2(1+rq)] $.

Further, similarly to Remark \ref{rem-1}, note that $\gamma_r>\frac{r\Re q-r_0}{12-r_0+r\Re q}>\frac{r\Re q-6}{r\Re q +6}\,.$ Therefore, we arrive at

\begin{corol}\label{ass6-sg-exten}
  The  semigroup generated by $G_r,$  $r\ge\frac6{\Re q}$, admits analytic extension with respect to the semigroup parameter to the sector of opening $\pi\,\frac{r\Re q-6}{r\Re q+6}.$
\end{corol}

\bigskip

\section{Concluding observations}\label{sec-discuss}
\setcounter{equation}{0}

1. All the results in this paper as well as those in the preceding work \cite{E-S-S} have been obtained under the assumption $f(0)=0$.

Let $f$ be a generator and $0\neq\tau\in\D$ be its null point. This point is the Denjoy--Wolff point for the semigroup  $\left\{u(t,\cdot)\right\} _{t\geq 0}$ 
generated by $f$ (under the condition that at least one of the semigroup elements is not an automorphism)  as well as for the resolvent family $\{G_r\}_{r>0}$, see \eqref{D-Wp}. One can conjugate the semigroup with the involution $m_\tau$ which maps $\tau$ to zero. The Denjoy--Wolff point of the semigroup 
$\left\{m_\tau\circ u(t,\cdot)\circ m_\tau \right\}_{t\ge0 }$ is zero. It is also possible to get an appropriate transformation for $f$ (see, for example, \cite{E-S-book}).

Unfortunately, we are not aware of any explicit transformation for nonlinear resolvents which enables one to move their Denjoy--Wolff point to zero. 
Because of this reason, all questions considered here are open when $0\neq\tau\in\D$.

Geometric properties of resolvents in the case where $\tau\in\partial\D$ have not been studied yet.

\medskip

2. It seems that our distortion and covering results (Theorems~\ref{th_posi1}--\ref{th_resol1}) are not sharp. Indeed, we do not know any example that shows their sharpness.

Observe that, by their definition, nonlinear resolvents are inverse functions. Therefore, we use the version of the inverse function theorem presented in \cite{E-S-2020a} 
to prove the distortion result. It may happen that a different method will allow one to improve our quantitative statement regarding distortion.

As for the covering theorem, it is based on the known covering result for hyperbolically convex functions. 
Since not every hyperbolically convex function is a resolvent, we expect that resolvents cover a disk of a radius larger than the one proven above.

\medskip

3.  We note that Questions 3--4 in Section~\ref{sect-intro} concerning  the {\it sharp} orders of starlikeness and spirallikeness are still open. 
Indeed, our approach includes a certain estimate of the range of the function $\frac{zp'(z)}{1+rp(z)}$ restricted to the disk of radius $\frac{3}{1+r\Re q}$. 
This estimation can be refined. Unfortunately, the method for obtaining such refinements we are aware of leads to very artificial formulas, absolutely `non-readable'. 
The problem, therefore, is to establish better results that will be appropriate for the subsequent use.

\medskip

4. Corollary~\ref{cor-tends_to_z} asserts that the net of normalized resolvents $\{(1+rq)G_r\}_{r>0}$ converges to the identity mapping  as $r\to \infty$, 
uniformly on compact subsets of~$\D$. At the same time, the following question is still open.
\begin{quest}
Does the net of normalized resolvents converge to the identity mapping uniformly on the whole open unit disk?
\end{quest}

If the answer to this question is affirmative, this immediately implies the result of Corollary~\ref{corr-boundary-fix}.

\medskip

5. Notice that the semigroup generated by $G_0=\Id$ is defined by $u(t,z)=e^{-t}z$ and can be analytically extended to the right half-plane with respect to the semigroup parameter. 
We have proved that for every $r\ge \frac{r_0}{\Re q}$, the semigroup generated by $G_r$ can be analytically extended to the sector of opening $\pi\gamma_r,$ see Theorem~\ref{thm-estim}. 
Keeping in mind that $\gamma_r$ tends to $1$ as $r\to0^+$, we conjecture:
\begin{conj}
  For every $r>0$, the semigroup generated by $G_r$ can be analytically extended to the sector $\displaystyle \left\{t\in\C: \left|\arg t - \arg(1+rq)\right| < \frac{ \pi\gamma_r}{ 2}\right\}.$
\end{conj}

\medskip

6. The study of dynamic and geometric properties of nonlinear resolvents in the multidimensional case has just begun.
In particular, in \cite{GHK2020} the authors considered nonlinear resolvents of generators on the open unit ball in $\C^n$ normalized by $f(0)=0,\ f'(0)=\Id$, and proved that:

$\bullet$ the family $\{G_r\}_{r>0}$ is an inverse L{\oe}wner chain;

$\bullet$ if $n = 2$, then the shearing of $(1 +r)G_r$ is quasi-convex of type A and also starlike of order $\frac45$.

$\bullet$ a sufficient condition for the nonlinear resolvents to admit a quasiconformal extension to $\C^n$ was obtained.

These results refer to multidimensional counterparts of Questions 2--4 in Section~\ref{sect-intro}, while the study of multivariate versions of
the other questions posed above is expected to be the matter of forthcoming research.

\medskip

\noindent{\bf Acknowledgment.} The authors are grateful to Guy Katriel and Elijah Liflyand for very helpful comments and discussions. The authors also thank the anonymous reviewer for valuable remarks.

\medskip

\begin{center}
  {\bf Declarations }\vspace{1mm}
\end{center}

\noindent{\bf Conflict of interest}  The authors declare that they have no conflict of interest. 

\noindent{\bf Data availability} This manuscript has no associated data.

\noindent{\bf Ethical Conduct} Not applicable.

\noindent{\bf Financial interests} The authors have no relevant financial or non-financial interests to disclose.


\bigskip




\begin{thebibliography}{950}

\bibitem{Ab-92} M. Abate,
The infinitesimal generators of semigroups of holomorphic maps, {\it Ann. Mat. Pura Appl.} {\bf161} (1992), 167--180.


\bibitem{A-C-P} C. Avecou, I. Chalendar and J. R. Partington,
Analyticity and compactness of semigroups of composition operators, {\it J. Math. Anal. Appl.} {\bf 437} (2016), 545--560.


\bibitem{B-P} E. Berkson and H. Porta,
Semigroups of analytic functions and composition operators, {\it Michigan Math. J.} \textbf{25} (1978), 101--115.


\bibitem{B-C-DM-book} F. Bracci, M.~D. Contreras and S. D\'{\i}az-Madrigal,
{\sl Continuous Semigroups of holomorphic self-maps of the unit disc}, Springer Monographs in Mathematics, Springer, 2020.


\bibitem{FSG} F. Bracci, M.~D. Contreras,  S. D\'{\i}az-Madrigal, M. Elin and D. Shoikhet,
Filtrations of infinitesimal generators, {\it Funct. Approx. Comment. Math.}  {\bf59} (2018), 99--115.










\bibitem{E-J-17} M. Elin and F. Jacobzon,
Analyticity of semigroups on the right half-plane, {\it J. Math. Anal. Appl.} {\bf448} (2017), 750--766.



\bibitem{EJ-coeff20} M. Elin and F. Jacobzon,
Coefficient body for nonlinear resolvents, {\it Annales UMCS}, {\bf 2} (2020), 41--53.


\bibitem{E-J-21a} M. Elin and F. Jacobzon, Families of inverse functions: coefficient bodies and the Fekete--Szeg\"{o} problem,  {\it Mediterr. J. Math.}  {\bf 19}   (2022),  https://doi.org/10.1007/s00009-022-02017-2.



\bibitem{EJ-est} M. Elin and F. Jacobzon,
Estimates on some functionals over non-linear resolvents, {\it Filomat} {\bf 37}, (2023), 797--808.

\bibitem{E-R-Sbook} M. Elin, S. Reich and D. Shoikhet,
\emph{Numerical Range of Holomorpic Mappings and Applications}, Birkh\"{a}user, Cham, 2019.


\bibitem{E-S-book} M. Elin and D. Shoikhet,
{\sl Linearization Models for Complex Dynamical Systems. Topics in univalent functions, functions equations and semigroup theory},
Birkh{\"a}user Basel, 2010.


\bibitem{E-S-2020a} M. Elin and D. Shoikhet,
A sharpened form of the inverse function theorem, {\it Annales UMCS}, {\bf LXXIII} (2019), 59--67.


\bibitem{E-S-S} M. Elin, D. Shoikhet and T. Sugawa,
Geometric properties of the nonlinear resolvent of holomorphic generators, {\it J. Math. Anal. Appl.} {\bf 483} (2020), No. 123614.



\bibitem{E-S-Ta11} M. Elin, D. Shoikhet and N. Tarkhanov, Separation of boundary singularities for holomorphic generators, {\it Annali Mat. Pura ed Appl.} {\bf 190} (2011), 595--618.

\bibitem{E-S-Ta} M. Elin, D. Shoikhet and N. Tarkhanov,
Analytic semigroups of holomorphic mappings and composition operators, {\it Comput. Methods Funct. Theory}, {\bf 18}  (2018), 269--294.


\bibitem{F-K-Z} M. Fait, J. G. Krzy\.{z} and J. Zygmunt,
Explicit quasiconformal extensions for some classes of univalent functions, {\it Comment. Math. Helv.} {\bf51} (1976), 279--285.

\bibitem{Good} A. W. Goodman, {\sl Univalent Functions} , Volume 1, Mariner Publishing Company, Inc., Florida, 1983.

\bibitem{G-K} I. Graham and G. Kohr,
{\sl Geometric Function Theory in One and Higher Dimensions}, Marcel Dekker, Inc, NY-Basel, 2003.

\bibitem{GHK2020} I. Graham, H. Hamada and G. Kohr, Loewner chains and nonlinear resolvents of the Carath\'{e}odory family on the unit ball in $\C^n$, {\it J. Math. Anal. Appl.} {\bf 491}  (2020), https://doi.org/10.1016/j.jmaa.2020.124289.

\bibitem{HKK2021} H. Hamada, G. Kohr, M. Kohr, The Fekete--Szeg\"o problem for starlike mappings and nonlinear resolvents of the Carath\'{e}odory family on the unit balls of complex Banach spaces, {\it Anal. Math. Phys.} {\bf 11} (2021), https://doi.org/10.1007/s13324-021-00557-6.






\bibitem{Ke-Me} F. R. Keogh and E. P. Merkes, A coefficient inequality for certain classes of analytic functions, {\it Proc Amer. Math. Soc.} {\bf 20} (1969), 8--12.



\bibitem{Ma-Mi-94} W. Ma and D. Minda, Hyperbolically convex functions, {\it Ann. Polon. Math.}, {\bf LX.1},  (1994), 81--100.

\bibitem{Me-Mi-91} D. Mej\'{i}a and D. Minda, Hyperbolic geometry in hyperbolically $K$-convex regions, {Rev. Colomb. Math.}  {\bf 25} (1991), 123--142.

\bibitem{Me-Pomm} D. Mej\'{i}a and Ch. Pommerenke, On hyperbolically convex functions, {\it J. Geom. Anal. } {\bf10} (2000), 365--378.


\bibitem{M-M} S. S. Miller and P. T. Mocanu,
\textsl{Differential Subordinations. Theory and Applications}, Marcel Dekker Inc., New York, 2000.




\bibitem{R-S-96}  S. Reich and D. Shoikhet,
Generation theory for semigroups of holomorphic mappings in Banach spaces, {\it Abstr. Appl. Anal.} {\bf 1} (1996), 1--44.

\bibitem{R-S1} S. Reich and D. Shoikhet,  {\sl Nonlinear Semigroups, Fixed Points, and the Geometry of  Domains in Banach Spaces}, World Scientific Publisher, Imperial  College Press, London, 2005.




\bibitem{SD} D. Shoikhet, {\sl Semigroups in Geometrical Function Theory}, Kluwer,
Dordrecht, 2001.


\bibitem{Su-12} T. Sugawa,
Quasiconformal extension of strongly spirallike functions, {\it  Comput. Methods Funct. Theory} {\bf 12} (2012), 19--30.



\end{thebibliography}
\end{document}